\documentclass[a4paper,11pt,fleqn]{article}

\usepackage{amsmath,amssymb,amsthm,backref}

\newcommand{\noprint}[1]{}

\newcommand{\diag}{\mathop{\rm diag}\nolimits}
\newcommand{\todo}[1][\null]{\ensuremath{\clubsuit}}

\newtheorem{theorem}{Theorem}
\newtheorem{lemma}{Lemma}

{\theoremstyle{definition} \newtheorem{definition}{Definition}

\newtheorem{remark}{Remark}

\begin{document}

\par\noindent {\LARGE\bf
Contractions with necessarily unbounded matrices\par}

{\vspace{4mm}\par\noindent 
Dmytro R. POPOVYCH
\par\vspace{2mm}\par}

{\vspace{2mm}\par\it
\noindent Faculty of Mechanics and Mathematics, National Taras Shevchenko University of Kyiv,\\
building 7, 2, Academician Glushkov prospectus, Kyiv, Ukraine, 03127
\par}

{\vspace{2mm}\par\noindent\rm E-mail: \it  deviuss@gmail.com
 \par}

{\vspace{7mm}\par\noindent\hspace*{5mm}\parbox{150mm}{\small
We prove that for each dimension not less than five 
there exists a contraction between solvable Lie algebras 
that can be realized only with matrices whose Euclidean norms necessarily approach infinity  
at the limit value of contraction parameter. 
Therefore, dimension five is the lowest dimension of Lie algebras between which contractions of the above kind exist.
}\par\vspace{5mm}}

\noprint{


\vspace{2mm}

}

\noprint{QR-decomposition
\begin{verbatim}
http://en.wikipedia.org/wiki/QR_decomposition
https://en.wikipedia.org/wiki/Matrix_decomposition#Example
\end{verbatim}
}

\section{Introduction}

The study of ways for implementing contractions between Lie algebras plays an important role 
in the theory of contractions from its very outset. 
The general notion of limiting processes between Lie algebra structures was first introduced by Segal~\cite{Segal1951},
who was inspired by examples of physical theories being a limit case of others.
Contracting Lie algebras and their representations became physicists' operating tool 
after the papers by In\"on\"u and Wigner~\cite{Inonu&Wigner1953, Inonu&Wigner1954}.
They intended to consider ({\it linear}) contractions whose matrices are linear in the contraction parameter 
but in fact they merely studied contractions that can be realized, in properly chosen bases of initial and target algebras,
by diagonal matrices with only zero and first powers of the contraction parameter on their diagonals. 
These particular linear contractions are called In\"on\"u--Wigner contractions, or briefly {\it IW-contractions}.
General linear contractions were more comprehensively analyzed by Saletan~\cite{Saletan1961}; 
hence they are sometimes called {\it Saletan contractions}.
Therein a rigorous definition of contraction in terms of the right action of the general linear group on Lie brackets 
was presented, becoming conventional in physical literature.
Another generalization of IW-contractions, where the diagonal elements are allowed to be real powers of the contraction parameter,
was suggested in \cite{Doebner&Melsheimer1967}. 
To realize these $p$-contractions, called also Doebner--Melsheimer contractions or, more often, {\it generalized IW-contractions}~\cite{Hegerfeldt1967}, 
it in fact suffices to use only integer powers of the contraction parameter~\cite{Popovych&Popovych2009}.

In the course of exploring possibilities for realizing contractions naturally arises a problem on existence
of contraction matrices that have well-defined (finite) limits 
at the limit value of the contraction parameter~\cite{Weimar-Woods2000}.
The analysis of the results on contractions of real and complex Lie algebras up to dimension four
\cite{Campoamor-Stursberg2008,Nesterenko&Popovych2006,Popovych&Popovych2010} shows 
that all of these contractions can be realized by such matrices.
Is the same true for Lie algebras of higher dimension?
The first study of this problem was carried out in~\cite{Weimar-Woods2000} 
for the contraction between two specially chosen five-dimensional Lie algebras.

Consider the $n$-dimensional ($n\geqslant 5$) solvable real Lie algebras $\mathfrak a$ and $\mathfrak a_0$ 
that are defined by the following nonzero commutation relations:
\begin{gather*}
\mathfrak a\colon \quad [e_1, e_3]=e_3, \quad [e_2, e_4]=e_4, \quad [e_1, e_2]=e_5,\\
\mathfrak a_0\colon \quad [e_1, e_3]=e_3, \quad [e_2, e_4]=e_4.
\end{gather*}
Using Mubarakzyanov's classification of low-dimensional Lie algebras~\cite{Mubarakzyanov1963b}, 
these algebras can be denoted by $A_{5.38}\oplus(n-5)A_1$ and $A_{2.1}\oplus A_{2.1}\oplus(n-4)A_1$. 
Note that each five-dimensional solvable Lie algebra with one-dimensional center and three-dimensional nilradical 
is isomorphic to either $A_{5.38}$ or $A_{2.1}\oplus A_{2.1}\oplus A_1$, 
and three is the minimal dimension of nilradical for five-dimensional solvable Lie algebras.
It is obvious that the contraction $\mathfrak a\to\mathfrak a_0$ is realized 
by the diagonal matrix $U=\diag(1, 1, 1, 1, \varepsilon^{-1}, 1, \dots, 1)$, 
whose fifth diagonal entry goes to infinity as $\varepsilon\to+0$. 
The same is true for the contraction $\bar{\mathfrak a}\to\bar{\mathfrak a}_0$ between 
the complexifications $\bar{\mathfrak a}$ and $\bar{\mathfrak a}_0$ of $\mathfrak a$ and $\mathfrak a_0$.
It was shown in~\cite{Weimar-Woods2000} 
that for $n=5$ any realization of the contraction $\mathfrak a\to\mathfrak a_0$ as a generalized In\"on\"u--Wigner contraction 
necessarily involves a negative power of the contraction parameter, 
and hence some entries of the corresponding contraction matrix approach infinity at zero. 
The purpose of the present paper is to prove the following stronger and more general assertion:

\begin{theorem}\label{TheoremOnSingulatityOfContractionMatrix}
The Euclidean norm of any contraction matrix that realizes the contraction of the algebra $\mathfrak a$ to the algebra $\mathfrak a_0$ 
approaches infinity at the limit point. 
The same is true for the complex counterpart of this contraction.
\end{theorem}

In other words, for any dimension $n\geqslant 5$ Theorem~\ref{TheoremOnSingulatityOfContractionMatrix} constructively gives a positive answer to the question
whether there exist contractions between $n$-dimensional Lie algebras 
that can be realized only by unbounded matrices, 
and dimension five is the lowest dimension for which contractions of the above kind exist.

We additionally show that, up to automorphisms of the algebra~$\mathfrak a$, 
the Euclidean norm of the tuple formed by the $(5,5)$th, \dots, $(5,n)$th entries of any contraction matrix in the chosen bases of the algebras $\mathfrak a$ and $\mathfrak a_0$ 
approaches infinity at the limit point of the contraction parameter. 
In particular, in the case $n=5$ it is the $(5,5)$th entry of a contraction matrix whose absolute value goes to infinity.  

Both Theorem~\ref{TheoremOnSingulatityOfContractionMatrix} and the last claim are directly extended to the complex case.

\section{Auxiliary results}
\label{SectionOnAuxiliaryResults}

Given a finite-dimensional vector space $V$ over the field~$\mathbb F=\mathbb R$ or~$\mathbb F=\mathbb C$, 
by $\mathcal L_n=\mathcal L_n(\mathbb F)$ we denote the set of all possible Lie brackets on~$V$, where $n=\dim V<\infty$.
Each element~$\mu$ of~$\mathcal L_n$ corresponds to a Lie algebra with the underlying space~$V$, $\mathfrak g=(V,\mu)$.
Fixing a basis $\{e_1,\dots,e_n\}$ of~$V$ leads to a bijection between $\mathcal L_n$ and
\[
\mathcal C_n=\{(c_{ij}^k)\in\mathbb F^{n^3}\mid c_{ij}^k+c_{ji}^k=0,\,
c_{ij}^{i'\!}c_{i'\!k}^{k'\!}+c_{ki}^{i'\!}c_{i'\!j}^{k'\!}+c_{jk}^{i'\!}c_{i'\!i}^{k'\!}=0\}.
\]
The structure constant tensor $(c_{ij}^k)\in\mathcal C_n$ associated with a Lie bracket $\mu\in\mathcal L_n$
is given by the formula $\mu(e_i,e_j)=c_{ij}^ke_k$.
Here and in what follows, the indices $i$, $j$, $k$, $i'$, $j'$ and $k'$ run from 1 to $n$
and the summation convention over repeated indices is assumed.
The right action of the group~${\rm GL}(V)$ on $\mathcal L_n$, which is conventional for the physical literature,
is defined~as
\[
(U\cdot\mu)(x,y)=U^{-1}\bigl(\mu(Ux,Uy)\bigr)\quad \forall U\in {\rm GL}(V),\forall \mu\in\mathcal L_n,\forall x,y\in V.
\]

\begin{definition}\label{DefOfContractions1}
Given a Lie bracket~$\mu\in\mathcal L_n$ and a continuous matrix function $U\colon (0,1]\to {\rm GL}(V)$, 
we construct the parameterized family of Lie brackets~$\mu_\varepsilon=U_\varepsilon\cdot\mu$, $\varepsilon \in (0,1]$.
Each Lie algebra $\mathfrak g_\varepsilon=(V,\mu_\varepsilon)$ is isomorphic to $\mathfrak g=(V,\mu)$.
If the limit 
\[
\lim\limits_{\varepsilon \to +0}\mu_\varepsilon(x,y)=
\lim\limits_{\varepsilon \to +0}U_\varepsilon{}^{-1}\mu(U_\varepsilon x,U_\varepsilon y)=:\mu_0(x,y)
\]
exists for any $x, y\in V$, then $\mu_0$ is a well-defined Lie bracket.
The Lie algebra $\mathfrak g_0=(V,\mu_0)$ is called a \emph{one-parametric continuous contraction}
(or simply a \emph{contraction}) of the Lie algebra~$\mathfrak g$.
We call a limiting process that provides $\mathfrak g_0$ from~$\mathfrak g$ with a matrix function 
a \emph{realization} of the contraction $\mathfrak g\to\mathfrak g_0$.
\end{definition}

The notion of contraction is extended to the case an arbitrary algebraically closed field
in terms of orbit closures in the variety of Lie brackets
\cite{Burde1999,Burde2005,Burde&Steinhoff1999,Grunewald&Halloran1988,Lauret2003}.

If a basis~$\{e_1, \ldots, e_n\}$ of~$V$ is fixed, 
then the operator $U_\varepsilon$ can be identified with its matrix $U_\varepsilon\in{\rm GL}_n(\mathbb F)$, 
which is denoted by the same symbol,
and Definition~\ref{DefOfContractions1} can be reformulated in terms of structure constants.
Let $C=(c^k_{ij})$ be the tensor of structure constants of the algebra~$\mathfrak g$ in the basis chosen.
Then the tensor $C_\varepsilon=({c}^k_{\varepsilon,ij})$ of structure constants of the algebra~$\mathfrak g_\varepsilon$ in this basis is the result of 
the action by the matrix~$U_\varepsilon$ on the tensor~$C$, 
$C_\varepsilon=C\circ U_\varepsilon$. 
In term of components this means that
\[c^k_{\varepsilon,ij}=(U_\varepsilon)^{i'}_i(U_\varepsilon)^{j'}_j(U_\varepsilon{}^{-1})^k_{k'}c^{k'}_{i'\!j'}.\]
Then Definition~\ref{DefOfContractions1} is equivalent to that the limit
\[\lim\limits_{\varepsilon\to+0}c^k_{\varepsilon,ij}=:c^k_{0,ij}\]
exists for all values of $i$, $j$ and $k$ and, therefore,
$c^k_{0,ij}$ are components of the well-defined structure constant tensor~$C_0$ of the Lie algebra~$\mathfrak g_0$.
The parameter $\varepsilon$ and the matrix-function $U_\varepsilon$ are called a \emph{contraction parameter} and a \emph{contraction matrix},
respectively.

Sequential contractions~\cite{Weimar-Woods2000} are defined analogously to continuous contractions 
using matrix sequences, $\{U_p,\,p\in\mathbb N\}\subset{\rm GL}(V)$, instead of continuous matrix functions.
For each Lie bracket from the sequence $\{\mu_p=U_p\cdot\mu,\,p\in\mathbb N\}$,
the Lie algebra $\mathfrak g_p=(V,\mu_p)$ is isomorphic to $\mathfrak g=(V,\mu)$.
If the limit
\[
\lim\limits_{p \to\infty}\mu_p(x,y)=
\lim\limits_{p \to\infty}U_p{}^{-1}\mu(U_p x,U_p y)=:\mu_0(x,y)
\]
exists for any $x, y\in V$, then $\mu_0$ is a well-defined Lie bracket on $V$.
The Lie algebra $\mathfrak g_0=(V,\mu_0)$ is called a \emph{sequential contraction} of the Lie algebra~${\mathfrak g}$.
Within the basis-dependent approach, each algebra $\mathfrak g_p$ is associated with the structure constant tensor $C_p=C\circ U_p$ with 
the components $c^k_{p,ij}=(U_p)^{i'}_i(U_p)^{j'}_j(U_p{}^{-1})^k_{k'}c^{k'}_{i'\!j'}$.
The existence of the above limit of~$\{\mu_p\}$ is equivalent to the existence of the limit
\[\lim\limits_{p\to\infty}c^k_{p,ij}=:c^k_{0,ij}\]
for all values of $i$, $j$ and $k$, where 
$c^k_{0,ij}$ are components of the structure constant tensor~$C_0$ of the Lie algebra~$\mathfrak g_0$.

Any continuous contraction from $\mathfrak g$ to $\mathfrak g_0$ gives an infinite family of matrix sequences resulting in 
sequential contractions from $\mathfrak g$ to $\mathfrak g_0$.
More precisely, if $U_\varepsilon$ is the matrix of the continuous contraction and
the sequence $\{\varepsilon_p,\, p\in\mathbb N\}$ satisfies the conditions $\varepsilon_p\in(0,1]$, $\varepsilon_p\to+0$, $p\to\infty$,
then еру matrix sequence $\{U_{\varepsilon_p},\, p\in\mathbb N\}$ generates a sequential contraction from $\mathfrak g$ to $\mathfrak g_0$. 

\noprint{
Conversely, if a sequence $\{U_p,\,p\in\mathbb N\}\subset{\rm GL}(V)$ defines a sequential contraction from $\mathfrak g$ to $\mathfrak g_0$ 
(and the sign of $\det U_p$ is the same for all $p\in\mathbb N$ if $\mathbb F=\mathbb R$)
then there exists a one-parametric continuous contraction from $\mathfrak g$ to $\mathfrak g_0$ with 
the associated continuous function $\tilde U\colon (0,1]\to {\rm GL}(V)$ such that $\tilde U_{1/p}=U_p$ for any $p\in\mathbb N$. 
The simple proof of this fact involves logarithms and exponents of matrices and, additionally,  
the polar decomposition in the real case. 
}

Definitions of special types of contractions, statements on properties and their proofs in the case of sequential contractions 
can be easily obtained via reformulation of those for the case of continuous contractions. 
It is enough to replace continuous parametrization by discrete one.
\noprint{
The replacement is invertible.
}

The following useful assertion is obvious. 

\begin{lemma}
If the matrix~$U_\varepsilon$ of a contraction $\mathfrak g\to\mathfrak g_0$ can be 
represented in the form $U_\varepsilon=\hat U_\varepsilon\check U_\varepsilon$, where 
$\hat U$ and $\check U$ are continuous functions from $(0,1]$ to ${\rm GL}_n(\mathbb F)$ and 
the function~$\check U$ has a limit~$\check U_0\in{\rm GL}_n(\mathbb F)$ at $\varepsilon \to +0$,
then $\hat U_\varepsilon\check U_0$ also is a matrix of the contraction $\mathfrak g\to\mathfrak g_0$.
\end{lemma}

The same is true for sequential contractions.
We will need a more particular lemma, which is related to the $LQ$ matrix decomposition 
and is in fact a computational counterpart of Proposition~1.7 from~\cite{Grunewald&Halloran1988} 
for the real and complex cases.

\begin{lemma}\label{LemmaOnTriAndOrthPartsOfContractionMatrix}
A Lie algebra $\mathfrak g$ is sequentially contracted to a Lie algebra $\mathfrak g_0$ if and only if in the fixed basis~$\{e_1, \ldots, e_n\}$ 
of the underlying space $V$ there exists the sequence $\{L_p,\,p\in\mathbb N\}$ 
of nondegenerate lower triangular $n\times n$ matrices and an orthogonal (resp.\ unitary) $n\times n$ matrix $Q$ in the real (resp.\ complex) case 
such that $C\circ L_p\to C_0\circ Q$ as $p\to\infty$. 
\end{lemma}

\begin{proof}
Using the sequential realization of contractions, we prove the lemma only for the real case 
since the complex case is considered in a similar way with replacing orthogonal matrices by unitary ones.
Let $\{U_p,\,p\in\mathbb N\}$ be a sequence of matrices that realize the contraction $\mathfrak g\to\mathfrak g_0$, 
i.e. $C\circ U_p\to C_0$, $p\to\infty$.
For each $p$, we decompose the matrix $U_p$ into triangular and orthogonal multipliers, $U_p=L_pQ_p$,
where $L_p$ is a lower triangular matrix and~$Q_p$ is an orthogonal matrix.
As the set of $n\times n$ orthogonal matrices is compact in the Euclidean topology, the sequence $\{Q_p,\,p\in\mathbb N\}$ contains
a convergent subsequence. 
Any subsequence of a matrix sequence realizes the same sequential contraction as the whole sequence.
Therefore, without loss of generality, we can assume that the sequence $\{Q_p\}$ itself is convergent. 
Its limit $Q_0$ is also an orthogonal matrix.
Since $C\circ U_p\to C_0$, $Q_p^{\,\rm T}\to Q_0^{\,\rm T}$ and the matrices $Q_p$ are orthogonal, 
\[C\circ L_p=C\circ L_p^{}Q_p^{}Q_p^{\,\rm T}=C\circ U_p^{}Q_p^{\,\rm T}\to C_0\circ Q_0^{\,\rm T}\] 
as $p\to\infty$.
We denote $Q_0^{\,\rm T}$ by~$Q$, completing the proof of the lemma.
\end{proof}

\begin{remark}\label{RemarkOnTriAndOrthPartsOfContractionMatrix}
The sequence of triangular matrices $\{L_p,\,p\in\mathbb N\}$ and the orthogonal matrix~$Q$ 
are defined in Lemma~\ref{LemmaOnTriAndOrthPartsOfContractionMatrix} up to the transformation
\[
\tilde L_p=M_pL_pD_p, \quad \tilde Q=KQD_0,
\]
where 
$K$ is the matrix of an orthogonal automorphism of~$\mathfrak g_0$, 
$D_0$ is a diagonal orthogonal (resp. unitary) matrix in the real (resp. complex) case, 
$M_p$ for each $p\in\mathbb N$ is the matrix of an automorphism of~$\mathfrak g$,
and the sequence of the triangular matrices $\{D_p,\,p\in\mathbb N\}$ approaches the matrix $D_0$.
\end{remark}

\section{Proof}
\label{SectionOnProof}

We prove Theorem~\ref{TheoremOnSingulatityOfContractionMatrix} in the real case.
For the complex case, orthogonal matrices should be replaced by unitary ones, and the other differences are indicated explicitly.

First we consider an arbitrary sequential realization of the contraction $\mathfrak a\to\mathfrak a_0$ with
a matrix sequence $\{U_p,\,p\in\mathbb N\}$. 
If we suppose that the Euclidean norm of $U_p$ does not approach infinity, then the sequence $\{U_p\}$
contains a bounded subsequence $\{U_{p_s},\,s\in\mathbb N\}$. 
Following the proof of Lemma~\ref{LemmaOnTriAndOrthPartsOfContractionMatrix},
we factorize each matrix $U_{p_s}$ into its lower triangular and orthogonal parts,
choose a subsequence of elements of $\{U_{p_s}\}$ with convergent orthogonal parts
and apply the algebraic limit theorem.
As a result, we construct a bounded sequence of lower triangular matrices and an orthogonal matrix $Q$, 
which satisfy Lemma~\ref{LemmaOnTriAndOrthPartsOfContractionMatrix} for $\mathfrak g=\mathfrak a$ and $\mathfrak g_0=\mathfrak a_0$.
At the same time, as we will see below, the sequence of Euclidean norms of such triangular matrices necessarily approaches infinity.
The contradiction obtained means that the Euclidean norm of $U_p$ approaches infinity. 

Suppose that there exists a continuous realization of the contraction $\mathfrak a\to\mathfrak a_0$ with
a continuous function $U\colon (0,1]\to {\rm GL}(V)$ 
for which the Euclidean norm of its values~$U_\varepsilon$ does not go to infinity as $\varepsilon\to+0$.
Then we can choose a sequence $\{\varepsilon_p,\,p\in\mathbb N\}\subset(0,1]$ such that 
its limit equals zero and the matrix sequence $\{U_{\varepsilon_p},\,p\in\mathbb N\}$ is bounded. 
As the last sequence realizes a sequential contraction $\mathfrak a\to\mathfrak a_0$, 
this immediately leads to a contradiction.

Given the above, it suffices to prove that for any sequence $\{L_p=(l^i_{p,j}),\,p\in\mathbb N\}$ of lower triangular matrices (and orthogonal matrix $Q=(q^i_j)$\,)
satisfying Lemma~\ref{LemmaOnTriAndOrthPartsOfContractionMatrix} for $\mathfrak g=\mathfrak a$ and $\mathfrak g_0=\mathfrak a_0$
the corresponding sequence of Euclidean norms goes to infinity. 

Let us look into the constraints on the matrix~$Q$. We denote the structure constant tensors of the algebras $\mathfrak a$ and $\mathfrak a_0$
in the chosen basis $\{e_1,\ldots,e_n\}$ of the underlying vector space by $C=(c_{ij}^k)$ and $C_0=(c_{0,ij}^k)$ respectively, 
Then $C_p=C\circ L_p$ and $\tilde C_0=C_0\circ Q$ are the structure constant tensors of the algebras $\mathfrak a_p$ and $\tilde{\mathfrak a}_0$ 
that are isomorphic to the algebras $\mathfrak a$ and $\mathfrak a_0$ with respect to the operators~$L_p$ and~$Q$.
By the construction, $\lim_{p\to\infty} c_{p,ij}^k=\tilde c_{0,ij}^k$.  
Since for any $i$, $j$, $k$ and $j^*=5,\dots,n$ we have that $c_{ij^*}^k=c_{ij}^1=c_{ij}^2=0$ and, for any $p$, $l^j_{p,i}=0$ if $i<j$, 
then $c_{p,ij^*}^k=c_{p,ij}^1=c_{p,ij}^2=0$ holds true for any $i$, $j$, $k$ and $p$. 
Hence the same is true for elements of~$\tilde C_0$, $\tilde c_{0,ij^*}^k=\tilde c_{0,ij}^1=\tilde c_{0,ij}^2=0$.  
At the same time, the corresponding components of $C_0$ also vanish by the definition of $\mathfrak a_0$. 
Geometrically, this means that $Q\langle e_5,\dots,e_n\rangle=\langle e_5,\dots,e_n\rangle$ and $Q\langle e_3, e_4\rangle\subset\langle e_3,\dots,e_n\rangle$. 
As the matrix~$Q$ is orthogonal, then it is a block diagonal matrix of the form
\begin{gather}\label{EqFormOfOrtLimitMatrixFor5DimSingularContraction}
Q=\begin{pmatrix}q^1_1&q^1_2\\q^2_1&q^2_2\end{pmatrix}\oplus\begin{pmatrix}q^3_3&q^3_4\\q^4_3&q^4_4\end{pmatrix}
\oplus\begin{pmatrix}q^{i^*}_{j^*}\end{pmatrix}, \quad\mbox{where}\quad i^*,j^*=5,\dots,n.
\end{gather}
There are three more values of the triplet $(i,j,k)$, namely $(1,4,3)$, $(2,4,3)$ and $(2,3,3)$,
for which the structure constants $c_{ij}^k$, $c_{p,ij}^k$ (for all values of~$p$) and hence $\tilde c_{ij}^k$ vanish.
In other words, we obtain the equations
\[
\begin{array}{ll}\arraycolsep=0ex
\tilde c_{14}^3=q^1_1q^3_3q^3_4+q^2_1q^4_3q^4_4=0,\quad &(q^1_1\bar q^3_3q^3_4+q^2_1\bar q^4_3q^4_4=0),
\\[1ex]
\tilde c_{24}^3=q^1_2q^3_3q^3_4+q^2_2q^4_3q^4_4=0,\quad &(q^1_2\bar q^3_3q^3_4+q^2_2\bar q^4_3q^4_4=0),
\\[1ex]
\tilde c_{23}^3=q^1_2(q^3_3)^2+q^2_2(q^4_3)^2=0,\quad &(q^1_2\bar q^3_3q^3_3+q^2_2\bar q^4_3q^4_3=0).
\end{array}
\]
In the brackets we present the corresponding equations for the complex case, and the bar denotes the complex conjugation.
Because of $q^1_1q^2_2-q^1_2q^2_1\ne0$, the first two equations imply that $q^3_3q^3_4=q^4_3q^4_4=0$. 
Combining the orthogonality of~$Q$ with the above equations gives the following two possibilities:

\medskip

1. $q^3_3=q^4_4=0$. Then $q^3_4q^4_3\ne0$, $q^1_1=q^2_2=0$ and $q^1_2q^2_1\ne0$.

\medskip

2. $q^3_3q^4_4\ne0$. Then $q^3_4=q^4_3=0$, $q^1_2=q^2_1=0$ and $q^1_1q^2_2\ne0$.

\medskip

The corresponding forms of the matrix $Q$ are
\[
Q=\begin{pmatrix}0&q^1_2\\q^2_1&0\end{pmatrix}\oplus\begin{pmatrix}0&q^3_4\\q^4_3&0\end{pmatrix}\oplus\begin{pmatrix}q^{i^*}_{j^*}\end{pmatrix}
\quad\mbox{and}\quad 
Q=\begin{pmatrix}q^1_1&0\\0&q^2_2\end{pmatrix}\oplus\begin{pmatrix}q^3_3&0\\0&q^4_4\end{pmatrix}\oplus\begin{pmatrix}q^{i^*}_{j^*}\end{pmatrix}.
\]
Recall that the matrix~$Q$ is defined up to the multiplication by the matrix of an orthogonal automorphism of~$\mathfrak a_0$ from the left
and by an orthogonal diagonal matrix from the right, cf. Remark~\ref{RemarkOnTriAndOrthPartsOfContractionMatrix}.
The change of the basis $(\tilde e_1, \tilde e_2, \tilde e_3, \tilde e_4, \tilde e_5,\dots, \tilde e_n)=(e_2, e_1, e_4, e_3, e_5,\dots, e_n)$, 
which is an orthogonal automorphism of the algebra~$\mathfrak a_0$, reduces the first case to the second one. 
In the second case the matrix~$Q$ can be made diagonal by the orthogonal automorphism 
\[
(\tilde e_1, \tilde e_2, \tilde e_3, \tilde e_4)=(e_1, e_2, e_3, e_4),\quad 
\tilde e_{j^*}=e_{i^*}q^{i^*}_{j^*}
\] 
of the algebra~$\mathfrak a_0$.
Therefore, it suffices to consider only the case of~$Q$ being the identity matrix., i.e., $\tilde C_0=C_0$.

For values of the triplet $(i,j,k)$ with $i,j,k=1,\dots,5$ that have not been used yet, we represent the conditions $\lim_{p\to\infty} c_{p,ij}^k=c_{0,ij}^k$  
in the form
\[
c_{p,ij}^k:=l^{i'}_{p,i}l^{j'}_{p,j}\hat l^k_{p,k'}c_{i'\!j'}^{k'}=c_{0,ij}^k+o_{p,ij}^k, 
\]
where $\hat L_p=(\hat l^i_{p,j})=L_p^{-1}$ denotes the inverse of the matrix~$L_p$
and $\lim_{p\to\infty}o_{p,ij}^k=0$.
The algebra~$\mathfrak a$ is the sum of the ideal spanned by the first five basis elements of this algebra and the abelian ideal spanned by the other basis elements. 
The matrix $L_p$ is lower triangular.
This is why the expressions for the structure constants $c_{p,ij}^k$ with $i,j,k=1,\dots,5$ 
do not involve entries~$l^i_{p,j}$ of the matrix~$L_p$ where $i>5$ or $j>5$.
As a result, we derive a system of equations on $l^i_{p,j}$ and $o_{p,ij}^k$ with $i,j,k=1,\dots,5$ 
(in what follows we generally omit the subscript~$p$ for concise presentation),
\begin{gather*}
l^1_1 = 1+o_{13}^3,
\quad
l^2_2 = 1+o_{24}^4, 
\quad
l^2_1 = o_{14}^4, 
\quad
l^1_1\frac{l^3_2}{l^3_3} = o_{12}^3, 
\quad
l^2_2\frac{l^4_3}{l^4_4} = o_{23}^4, 
\quad
-l^2_2\frac{l^5_4}{l^5_5} = o_{24}^5, 
\\
-l^2_2\frac{l^4_1}{l^4_4}+l^2_1\frac{l^4_2}{l^4_4}-l^1_1\frac{l^3_2}{l^3_3}\frac{l^4_3}{l^4_4} = o_{12}^4, 
\quad
-l^1_1\frac{l^5_3}{l^5_5}+(l^1_1-l^2_1)\frac{l^4_3}{l^4_4}\frac{l^5_4}{l^5_5} = o_{13}^5, 
\\
-l^2_1\frac{l^5_4}{l^5_5} = o_{14}^5, 
\quad
-l^2_2\frac{l^4_3}{l^4_4}\frac{l^5_4}{l^5_5} = o_{23}^5, 
\quad
-(l^1_1-l^2_1)\frac{l^4_3}{l^4_4} = o_{13}^4, 
\\
\frac{l^1_1l^2_2}{l^5_5}-l^1_1\frac{l^3_2}{l^3_3}\frac{l^5_3}{l^5_5}-\left(
-l^2_2\frac{l^4_1}{l^4_4}+l^2_1\frac{l^4_2}{l^4_4}-l^1_1\frac{l^3_2}{l^3_3}\frac{l^4_3}{l^4_4}
\right)\frac{l^5_4}{l^5_5} = o_{12}^5. 
\end{gather*}
We solve the equations in the first two rows with respect to 
$l^3_2$, $l^4_3$, $l^5_4$, $l^4_1$ and $l^5_3$ and substitute the obtained expressions 
into the last equation, which gives 
\[
\frac{l^1_1l^2_2}{l^5_5} = o_{12}^5-\frac{o_{24}^5}{l^2_2}o_{12}^4
-\left(o_{13}^5+\frac{l^1_1-l^2_1}{(l^2_2)^2}o_{23}^4o_{24}^5\right)\frac{o_{12}^3}{l^1_1}. 
\]
The last equality obviously implies 
that $l^1_{p,1}l^2_{p,2}/l^5_{p,5}\to0$, i.e., $|l^5_{p,5}|\to\infty$ as $p\to\infty$.
Therefore, the sequence of Euclidean norms of the matrices~$L_p$, $p\in\mathbb N$, also goes to infinity. 
Note that the equations in the third row of the system do not lead to additional constraints for entries of~$L_p$, 
and the sixth and eight equations imply that $l^5_{p,4}/l^5_{p,5}\to0$ and  $l^5_{p,3}/l^5_{p,5}\to0$ as $p\to\infty$.

Now we additionally show that, up to automorphisms of the algebra~$\mathfrak a$, 
the Euclidean norm of the tuple formed by $(5,5)$th, \dots, $(5,n)$th entries of any contraction matrix 
in the chosen bases of the algebras $\mathfrak a$ and $\mathfrak a_0$ 
goes to infinity at the limit point of the contraction parameter. 

Given a sequential contraction $\mathfrak a\to\mathfrak a_0$ with a matrix sequence $\{U_p,\,p\in\mathbb N\}$, 
we again factorize each matrix~$U_p$ into its lower triangular and orthogonal parts $L_p$ and $Q_p$, $U_p=L_pQ_p$.
As the limit of any convergent subsequence of $\{Q_p,\,p\in\mathbb N\}$ 
has the form~\eqref{EqFormOfOrtLimitMatrixFor5DimSingularContraction}, 
for each such subsequence and hence for the entire sequence $\{Q_p,\,p\in\mathbb N\}$ 
we have that $q^i_{p,j}\to0$ as $p\to\infty$ if $i=1,\dots,4$ and $j=5,\dots,n$ or if $i=5,\dots,n$ and $j=1,\dots,4$.
For the corresponding subsequences of $\{L_p,\,p\in\mathbb N\}$ 
the limits $|l^5_{p,5}|\to\infty$, $l^5_{p,4}/l^5_{p,5}\to0$ and $l^5_{p,3}/l^5_{p,5}\to0$ as $p\to\infty$ hold true. 
Hence, the same limits hold true for the whole sequence $\{L_p,\,p\in\mathbb N\}$ (otherwise, we obtain a contradiction).
Using Remark~\ref{RemarkOnTriAndOrthPartsOfContractionMatrix}, for each~$p$ we multiply the matrix~$L_p$ from the left by 
the matrix 
\[
M_p=E-\frac1{l^1_{p,1}}\left(l^5_{p,1}-\frac{l^2_{p,1}}{l^2_{p,2}}l^5_{p,2}\right)E^5_1-\frac{l^5_{p,2}}{l^2_{p,2}}E^5_2,
\]
which is associated with an automorphism of~$\mathfrak a$.
Here $E$ denotes the $n\times n$ identity matrix and
$E^i_j$ denotes the $n\times n$ matrix with the unit entry on the cross of the $i$-th row and the $j$-th column and zero otherwise.
The entries $\tilde l^5_{p,1}$ and $\tilde l^5_{p,2}$ of the matrix $\tilde L_p=M_pL_p$ are equal to zero. 
Then for the $(5,j)$th entries of the matrix $\tilde U_p=\tilde L_pQ_p=M_pU_p$ with $j\geqslant 5$ we have
\begin{gather*}
\lim_{p\to\infty}\sum_{j=5}^n\left((\tilde U_p)^5_j\right)^2=
\lim_{p\to\infty}\sum_{j=5}^n\left(\tilde l^5_{p,3}q^3_{p,j}+\tilde l^5_{p,4}q^4_{p,j}+\tilde l^5_{p,5}q^5_{p,j}\right)^2
\\\qquad{}
=\lim_{p\to\infty}(\tilde l^5_{p,5})^2\sum_{j=5}^n\left(\frac{\tilde l^5_{p,3}}{\tilde l^5_{p,5}}q^3_{p,j}+\frac{\tilde l^5_{p,4}}{\tilde l^5_{p,5}}q^4_{p,j}+q^5_{p,j}\right)^2
=\lim_{p\to\infty}(\tilde l^5_{p,5})^2\sum_{j=5}^n\left(q^5_{p,j}\right)^2
\\\qquad{}
=\lim_{p\to\infty}(\tilde l^5_{p,5})^2
=\infty.
\end{gather*}
We additionally use the facts that $\sum_{j=1}^nq^5_{p,j}q^5_{p,j}=1$ and $q^5_{p,j}\to 0$ as $p\to\infty$ if $j<5$.

The proof for the case of continuous contractions is similar. 
The only additional feature is continuity with respect to the contraction parameter~$\varepsilon$. 
The Gram--Schmidt process applied to the contraction matrix~$U_\varepsilon$ leads to a factorization  
in which both the lower triangular and orthogonal parts~$L_\varepsilon$ and~$Q_\varepsilon$ are continuous matrix-functions of~$\varepsilon$. 
Then the corresponding automorphism~$M_\varepsilon$ of~$\mathfrak a$ that annuls the $(5,1)$th and $(5,2)$th entries of~$L_\varepsilon$ 
is also continuous with respect to~$\varepsilon$, which implies the continuity of $\tilde U_\varepsilon=M_\varepsilon U_\varepsilon$.

\section{Conclusion}

We have constructed a single example of the solvable Lie algebras $\mathfrak a$ and $\mathfrak a_0$ for each dimension greater than four 
such that the contraction $\mathfrak a\to\mathfrak a_0$ cannot be realized by a bounded matrix-function. 
Moreover, we have showed that, up to automorphisms of the algebras $\mathfrak a$ and $\mathfrak a_0$, 
the Euclidean norm of the tuple formed by $(5,5)$th, \dots, $(5,n)$th entries of any contraction matrix 
in the chosen bases of the algebras $\mathfrak a$ and $\mathfrak a_0$ 
necessarily approaches infinity at the limit point of the contraction parameter. 

The proof of Theorem~\ref{TheoremOnSingulatityOfContractionMatrix} involves several techniques.
The first step in managing the contraction matrix is to factorize it into lower triangular and orthogonal parts 
and then apply Lemma~\ref{LemmaOnTriAndOrthPartsOfContractionMatrix} in order to move the orthogonal part from under the limit 
to the contracted structure constants. 
Due to the special structure of the considered Lie algebras it is possible to prove 
that the orthogonal part is an automorphism matrix of the contracted algebra~$\mathfrak a_0$ 
and hence can be set to the identity matrix, which is neglected.
For each fixed $(i,j,k)$ we consider the difference between the corresponding transformed and contracted structure constants
as a new unknown value, which should approach zero. 
This reduces the limit relations between the structure constants 
to the system of algebraic equations in entries of the lower triangular part and new vanishing values.
For the completion of the proof, it suffices to find out that the obtained algebraic equations for $i,j,k=1,\dots, 5$
involve only entries of the lower triangular part and new vanishing values with indices that run in the same range. 
The algebraization of the limit relations between the structure constants 
and considering a subsystem of algebraic equations that does not depend on the dimension~$n$ 
allow for the verification of all computations using a computer program.

It is not understandable yet what properties lead to the above phenomenon, which does not appear in lower dimensions.
We can only note that in the case $n=5$ the contraction $\mathfrak a\to\mathfrak a_0$ is direct, i.e. there is no intermediate algebra $\tilde{\mathfrak a}_0$ 
such that $\mathfrak a\to\tilde{\mathfrak a}_0$ and $\tilde{\mathfrak a}_0\to\mathfrak a_0$ are well-defined proper contractions.
This follows from the fact that the derivation algebras of $\mathfrak a$ and $\mathfrak a_0$ are of dimensions six and seven, respectively,
and any contraction leads to the increase of the dimension of the derivation algebra.

Since this is the first example in the literature, it is not clear how common are contractions with necessarily unbounded contraction matrices.
At the same time, we have no reason to assume the above phenomenon unique, and we could guess that the number of such contractions grows 
when dimension of Lie algebras increases.

We can pose one more problem related to the subject considered. 
Given a generalized IW-contraction that necessarily involves negative powers of the contraction parameter,
does there exist a realization of this contraction with a bounded matrix-function of another kind?

\section*{Acknowledgements}
The author is grateful to Dietrich Burde, Anatoliy Petravchuk and Roman Popovych for helpful discussions.


\begin{thebibliography}{99}\itemsep=0.1ex
\footnotesize

\bibitem{Burde1999}
Burde D.,
Degenerations of nilpotent Lie algebras,
{\it J. Lie Theory}  {\bf 9} (1999), 193--202.

\bibitem{Burde2005}
Burde D.,
Degenerations of 7-dimensional nilpotent Lie algebras,
{\it Comm. Algebra} {\bf 33} (2005), 1259--1277; arXiv:math.RA/0409275.

\bibitem{Burde&Steinhoff1999}
Burde D. and Steinhoff C.,
Classification of orbit closures of 4-dimensional complex Lie algebras,
{\it J.~Algebra}, {\bf 214} (1999), 729--739.

\bibitem{Campoamor-Stursberg2008}
Campoamor-Stursberg R.,
Some comments on contractions of Lie algebras,
{\it Adv. Studies Theor. Phys.},  2008, {\bf 2}, 865--870.

\bibitem{Doebner&Melsheimer1967}
Doebner H.D. and Melsheimer O.,
On a class of generalized group contractions,
{\it Nuovo Cimento A(10)}, {\bf 49} (1967), 306--311.

\bibitem{Grunewald&Halloran1988}
Grunewald F. and O'Halloran J.,
Varieties of nilpotent Lie algebras of dimension less than six,
{\it J. Algebra}, {\bf 112} (1988), 315--325.

\bibitem{Hegerfeldt1967}
Hegerfeldt G.C.,
Some properties of a class of generalized In\"on\"u--Wigner contractions,
{\it Nuovo Cimento A(10)}, {\bf 51} (1967), 439--447.

\bibitem{Inonu&Wigner1953}
In\"on\"u E. and Wigner E.P.,
On the contraction of groups and their representations,
{\it Proc. Nat. Acad. Sci. U.S.A.}, {\bf 39} (1953), 510--524.

\bibitem{Inonu&Wigner1954}
In\"on\"u E. and Wigner E.P.,
On a particular type of convergence to a singular matrix,
{\it Proc. Nat. Acad. Sci. U.S.A.}, {\bf 40} (1954), 119--121.

\bibitem{Lauret2003}
Lauret J.,
Degenerations of Lie algebras and geometry of Lie groups,
{\it Differential Geom. Appl.}, {\bf 18} (2003), 177--194.

\bibitem{Mubarakzyanov1963b}
Mubarakzyanov~G.M. 
The classification of the real structure of five-dimensional Lie algebras,
{\it Izv. Vys. Ucheb. Zaved. Matematika}, 1963, no.~3 (34), 99--106 (in Russian).

\bibitem{Nesterenko&Popovych2006}
Nesterenko M. and Popovych R.O.,
Contractions of low-dimensional Lie algebras,
{\it J. Math. Phys.}, {\bf 47} (2006), 123515, 45 pp; arXiv:math-ph/0608018.

\bibitem{Popovych&Popovych2009}
Popovych D.R. and Popovych R.O.,
Equivalence of diagonal contractions to generalized IW-contractions with integer exponents,  
{\it Linear Algebra Appl.}, {\bf 431} (2009), 1096--1104; arXiv:0812.4667.

\bibitem{Popovych&Popovych2010}
Popovych D.R. and Popovych R.O.,
Lowest dimensional example on non-universality of generalized In\"on\"u--Wigner contractions,
{\it J. Algebra} {\bf 324} (2010), 2742--2756, arXiv:0812.1705. 

\bibitem{Saletan1961}
Saletan E. J.,
Contraction of Lie groups,
{\it J. Math. Phys.}, {\bf 2} (1961), 1--21.

\bibitem{Segal1951}
Segal I.E.,
A class of operator algebras which are determined by groups,
{\it Duke Math. J.}, {\bf 18} (1951), 221--265.

\bibitem{Weimar-Woods2000}
Weimar-Woods E.,
Contractions, generalized In\"on\"u--Wigner contractions and deformations of finite-dimensional Lie algebras,
{\it Rev. Math. Phys.}, {\bf 12} (2000), 1505--1529.

\end{thebibliography}
\end{document}